\documentclass[12pt,bezier]{article}
\usepackage{times}
\usepackage{booktabs}
\usepackage{appendix}
\usepackage{pifont}
\usepackage{floatrow}
\usepackage{caption}
\usepackage{mathrsfs}
\usepackage[fleqn]{amsmath}
\usepackage{amsfonts,amsthm,amssymb,mathrsfs,bbding}
\usepackage{txfonts}
\usepackage{graphics,multicol}
\usepackage{graphicx}
\usepackage{color}
\usepackage{caption}
\usepackage{cite}
\usepackage{latexsym,bm}
\usepackage{longtable}
\usepackage{supertabular}
\usepackage{enumerate}
\usepackage{amsmath}
\evensidemargin=\oddsidemargin\topmargin=-1.5cm

\newtheorem{lem}{Lemma}[section]

\newtheorem{thm}{Theorem}[section]
\newtheorem{cor}{Corollary}[section]

\newtheorem{remark}{Remark}

\theoremstyle{definition}

\addtocounter{section}{0}
\allowdisplaybreaks[4]
\begin{document}
\title{Inertia indices of signed graphs with given cyclomatic number and given number of pendant vertices  \footnote{This work is sponsored by Natural Science Foundation of Xinjiang Uygur Autonomous
Region (No. 2022D01A218).}}
\author{{Jie Pu,\ \   Fang Duan \footnote{Email: fangbing327@126.com}}\\[2mm]
\small School of Mathematics Science,  Xinjiang Normal University,\\
\small Urumqi, Xinjiang 830017, P.R. China}
\date{}
\maketitle {\flushleft\large\bf Abstract:} Let $\Gamma=(G, \sigma)$ be a signed graph of order $n$ with underlying graph $G$ and a sign function $\sigma: E(G)\rightarrow \{+, -\}$. Denoted by $i_+(\Gamma)$, $\theta(\Gamma)$ and $p(\Gamma)$ the positive inertia index, the cyclomatic number and the number of pendant vertices of $\Gamma$, respectively. In this article, we prove that $i_+(\Gamma)$, $\theta(\Gamma)$ and $p(\Gamma)$ are related by the inequality $i_+(\Gamma)\geq \frac{n-p(\Gamma)}{2}-\theta(\Gamma)$. Furthermore, we completely characterize the signed graph $\Gamma$ for which $i_+(\Gamma)=\frac{n-p(\Gamma)}{2}-\theta(\Gamma)$. As a by-product, the inequalities $i_-(\Gamma)\geq \frac{n-p(\Gamma)}{2}-\theta(\Gamma)$ and $\eta(\Gamma)\leq p(\Gamma)+2\theta(\Gamma)$ are also obtained, respectively.

\begin{flushleft}
\end{flushleft}

\section{Introduction}
Let $G=(V(G), E(G))$ be a simple graph with vertex set $V(G)=\{v_1, v_2, \ldots, v_n\}$ and edge set $E(G)$. $|V(G)|=n$ is called the \emph{order} of $G$. The \emph{cyclomatic number} of $G$, written as $\theta(G)$, is defined as $\theta(G)=|E(G)|-|V(G)|+c(G)$, where $c(G)$ denotes the number of connected components of $G$.
We say that a connected graph $G$ is a \emph{tree} if $\theta(G)=0$ and a \emph{unicyclic graph} if $\theta(G)=1$. A graph $H$ is called a \emph{subgraph} of $G$ if $V(H)\subseteq V(G)$ and $E(H)\subseteq E(G)$. Further, $H$ is called an \emph{induced subgraph} of $G$ if two vertices of $V(H)$ are adjacent in $H$ if and only if they are adjacent in $G$. For a subset $U$ of $V(G)$, denoted by $G-U$ the induced subgraph of $G$ obtained from $G$ by removing all vertices of $U$ together with all edges incident to them. In particular, we use $G-x$ or $G-x-y$ instead of $G-U$ if $U=\{x\}$ or $U=\{x, y\}$. We define $N_G{(x)}=\{y\in V(G): xy\in E(G)\}$ the \emph{neighbor set} of $x$ in $G$, and $|N_G(x)|$ is called the \emph{degree} of $x$ in $G$, which is written as $d_G(x)$ or $d(x)$. A vertex of $G$ is said to be \emph{pendant} if it has degree 1, and denoted by $p(G)$ the number of pendant vertices in $G$. As usual, $C_n$ and $P_n$ is referred to the \emph{cycle} and \emph{path} with $n$ vertices, respectively.

The \emph{adjacency matrix} $A(G)=(a_{ij})_{n\times n}$ of $G$ is defined as follows: $a_{ij}=1$ if $v_i$ is adjacent to $v_j$, and $a_{ij}=0$ otherwise. The number of positive and negative eigenvalues of $A(G)$ are called \emph{positive inertia index} and \emph{negative inertia index} of $G$, which are denoted by $i_+(G)$ and $i_-(G)$, respectively. The \emph{rank} $r(G)$ of $G$ is the number of non-zero eigenvalues of $A(G)$, that is, $r(G)=i_-(G)+i_+(G)$, and the \emph{nullity} $\eta(G)$ of $G$ is the number of zero eigenvalues of $A(G)$. Evidently, $r(G)+\eta(G)=n$. As the field of interdisciplinary research continues to evolve, the study of inertia indices and nullity has also been developing in an interdisciplinary direction. For example, in chemical graph theory and algebraic graph theory, inertia indices and nullity are used to analyze the chemical properties of molecules and the algebraic structure of graphs. Hence the invariants $i_+(G)$, $i_-(G)$ and $\eta(G)$ have been broadly investigated in the last decades. For some results on this topic, we refer the reader to \cite{FanY.Z, GuiHai.Yu, H.Ma,A.Torgasev}, and references therein.

A signed graph $\Gamma=(G, \sigma)$ consists of the underlying graph $G$ and a sign function $\sigma: E(G)\rightarrow \{+, -\}$. The \emph{adjacency matrix} of $\Gamma=(G, \sigma)$ is defined as $A(\Gamma)=(a_{u,v}^\sigma)_{u,v\in V(G)}$, where $a_{u,v}^\sigma=\sigma(uv) \cdot 1$ if $uv\in E(G)$ and $a_{u,v}^\sigma=0$ otherwise. The inertia indices, nullity and other basic notations of signed graphs are similarly to that of simple graphs. The sign of a cycle $C^\sigma$ is defined by $sgn(C^\sigma)=\prod_{e\in E(C)}\sigma(e)$. If $sgn(C^\sigma)=+$ (resp., $sgn(C^\sigma)=-$), then $C^\sigma$ is said to be positive (resp., negative). A signed graph $\Gamma$ is said to be \emph{balanced} if all of its cycles are positive, and \emph{unbalanced} otherwise. In particular, an acyclic signed graph is balanced.

Signed graphs were introduced by Harary \cite{Harary.F} in connection with the study of the theory of social balance in social psychology (see \cite{B.Deradass}). In particular, structural balance dynamics under perceived sentiment is an important research direction(see \cite{Shang.Y}).
And it also has extensive applied research in other fields: In software engineering, software objectives can be represented as vertices in signed graphs, with positive and negative dependency relationships between objectives denoted by edges, and clustering algorithms are used to partition the objective graph for optimizing component development. In the field of data clustering, nodes represent data objects, and the positive or negative signs of edges indicate similarity or correlation, enabling data clustering, among others.
In the study of spectral graph theory, Many results related to simple graphs cannot be directly extended to signed graphs. For more background on signed graphs we refer to \cite{F.Belardo}. Recently, there have been a number of investigations on the spectra of signed graphs, see \cite{F.Duan1,W.H.Haemers, Q.Wu} and references therein.

For a simple graph $G$, the relations between $\eta(G)$, $\theta(G)$, and $p(G)$ are established in \cite{Xiaobin.Ma}, which spire us to consider relations about $i_+(\Gamma)$, $\theta(\Gamma)$ and $p(\Gamma)$ for a signed graph $\Gamma$. In this paper, we prove that $i_+(\Gamma)\geq \frac{n-p(\Gamma)}{2}-\theta(\Gamma)$ for a signed graph $\Gamma$ of order $n\geq 2$. As a by-product, $i_-(\Gamma)\geq \frac{n-p(\Gamma)}{2}-\theta(\Gamma)$ is also derived. By using these two inequalities of $i_+(\Gamma)$ and $i_-(\Gamma)$, the main results about nullity in \cite{Xiaobin.Ma} can be obtained because of $\eta(\Gamma)=n-(i_+(\Gamma)+i_-(\Gamma))$. Furthermore, we characterize the corresponding extremal signed graphs for $i_+(\Gamma)=\frac{n-p(\Gamma)}{2}-\theta(\Gamma)$ and $i_-(\Gamma)=\frac{n-p(\Gamma)}{2}-\theta(\Gamma)$, respectively.

\section{Elementary Lemmas}
In this section, we list some notions and lemmas for the latter use.

\begin{thm}[\cite{D.Cvetkovi}]\label{thm-2-0}
(Interlacing Theorem) Let $Q$ be a real $n\times m$ matrix such that $Q^TQ=I_m$ (where $m < n$), and let $A$ be an $n\times n$ real symmetric matrix
with eigenvalues $\lambda_1\geq \lambda_2\geq \cdots\geq \lambda_n$. If the eigenvalues of $B=Q^TAQ$ are $\mu_1\geq \mu_2\geq \cdots\geq \mu_m$, then
the eigenvalues of $B$ interlace those of $A$, that is, $\lambda_{n-m+i}\leq \mu_i\leq \lambda_i (i=1,2, ..., m)$.
\end{thm}

Lemma \ref{lem-2-1} can be easily deduced from Interlacing Theorem.

\begin{lem}\label{lem-2-1}
Let $\Gamma'$ be an induced signed subgraph of signed graph $\Gamma$. Then $i_+(\Gamma)\geq i_+(\Gamma')$ and $i_-(\Gamma)\geq i_-(\Gamma')$.
\end{lem}

\begin{lem}[\cite{G.H.Yu}]\label{lem-2-2}
Let $\Gamma$ be a signed graph containing a pendant vertex $u$, and let $\Gamma'$ be the induced signed subgraph of $\Gamma$ obtained by deleting $u$ together with the vertices adjacent to it. Then $i_+(\Gamma)=i_+(\Gamma')+1$ and $i_-(\Gamma)=i_-(\Gamma')+1$.
\end{lem}

Lemma \ref{lem-2-3} is well-know and the proof is omitted.

\begin{lem}\label{lem-2-3}
Let signed graph $\Gamma=\Gamma_1\cup \Gamma_2\cup \cdots\cup \Gamma_t$, where $\Gamma_i$ $(i=1, 2,\ldots, t)$ are connected signed components of $\Gamma$. Then $i_+(\Gamma)=\sum^t _{i=1} i_+(\Gamma_i)$ and $i_-(\Gamma)=\sum^t _{i=1} i_-(\Gamma_i)$.
\end{lem}

For a simple graph, the following lemma is derived in \cite{Xiaobin.Ma}, we present the proof here for signed graphs by a similar manner.
\begin{lem}\label{lem-2-4}
Let $x$ be any vertex of a connected signed graph $\Gamma$ and $s$ be the number of signed components of $\Gamma-x$. Then
\begin{enumerate}[(i)]
\item $d(x)\geq m+s-r$, where $m$ is the number of 2-degree neighbors of $x$, and $r$ is the number of components containing 2-degree neighbors of $x$;
\item $\theta(\Gamma-x)=\theta(\Gamma)-d(x)+s$.
\end{enumerate}
\end{lem}

\begin{proof}
Let $x\in V(\Gamma)$ and $\Gamma-x=H_1^\sigma\cup H_2^\sigma\cup\cdots\cup H_s^\sigma$, where $H_k^\sigma$ $(k=1,\ldots,s)$ are all connected signed components of $\Gamma-x$. We arrange these components such that $H_i^\sigma$ contains $2$-degree neighbors of $x$ for $i=1,\cdots,r$, and $H_j^\sigma$ contains no $2$-degree neighbors of $x$ for $j=r+1,\cdots,s$. Then
$\sum_{i=1}^{r}|N_{H_i^\sigma}(x)|\ge m\ge r$, and $\sum_{j=r+1}^{s}|N_{H_j^\sigma}(x)|\ge s-r$. Therefore,
$$d(x)=\sum_{i=1}^{r}|N_{H_i^\sigma}(x)|+\sum_{j=r+1}^{s}|N_{H_j^\sigma}(x)|\geq m+s-r,$$ so (i) is derived.

The following equality is consequences of the definition of $\theta(\Gamma)$,
\begin{align}
\theta(\Gamma-x)=|E(\Gamma-x)|-|V(\Gamma-x)|+s&=(|E(\Gamma)|-d(x))-(|V(\Gamma)|-1)+s\nonumber\\
&=(|E(\Gamma)|-|V(\Gamma)|+1)-d(x)+s\nonumber\\
&=\theta(\Gamma)-d(x)+s\nonumber.
\end{align}
Hence (ii) holds.
\end{proof}

\begin{lem}[\cite{GuiHai.Yu}, \cite{G.H.Yu}, \cite{G.H.Yu1}]\label{lem-2-5}
Let $C_n^\sigma$, $P_n^\sigma$ be a signed cycle, a signed path of order $n$, respectively.
\begin{enumerate}[(i)]
\item If $C_n^\sigma$ is balanced, then $$i_+(C_n^\sigma)=\left\{
        \begin{array}{cccccc}
        \frac{n}{2}-1,\ n\equiv\ 0(\bmod\ 4);\\
        \frac{n+1}{2},\ n\equiv\ 1(\bmod\ 4);\\
        \frac{n}{2},\ n\equiv\ 2(\bmod\ 4);\\
        \frac{n-1}{2},\ n\equiv\ 3(\bmod\ 4).
\end{array}
      \right.
i_-(C_n^\sigma)=\left\{
        \begin{array}{cccccc}
        \frac{n}{2}-1,\ n\equiv\ 0(\bmod\ 4);\\
        \frac{n-1}{2},\ n\equiv\ 1(\bmod\ 4);\\
        \frac{n}{2},\ n\equiv\ 2(\bmod\ 4);\\
        \frac{n+1}{2},\ n\equiv\ 3(\bmod\ 4).\\
\end{array}
      \right.$$

\item  If $C_n^\sigma$ is unbalanced, then $$i_+(C_n^\sigma)=\left\{
        \begin{array}{cccccc}
        \frac{n}{2},\ n\equiv\ 0(\bmod\ 4);\\
        \frac{n-1}{2},\ n\equiv\ 1(\bmod\ 4);\\
        \frac{n}{2}-1,\ n\equiv\ 2(\bmod\ 4);\\
        \frac{n+1}{2},\ n\equiv\ 3(\bmod\ 4).
\end{array}
      \right.
i_-(C_n^\sigma)=\left\{
        \begin{array}{cccccc}
        \frac{n}{2},\ n\equiv\ 0(\bmod\ 4);\\
        \frac{n+1}{2},\ n\equiv\ 1(\bmod\ 4);\\
        \frac{n}{2}-1,\ n\equiv\ 2(\bmod\ 4);\\
        \frac{n-1}{2},\ n\equiv\ 3(\bmod\ 4).\\
\end{array}
      \right.$$
\item $i_+(P_n^\sigma)=i_-(P_n^\sigma)=\lfloor\frac{n}{2}\rfloor$.
\end{enumerate}
\end{lem}

\section{Positive inertia index of signed graphs with given cyclomatic number and given number of pendant vertices}
In this section, the relations between positive inertia index of a signed graph $\Gamma$ and $\theta(\Gamma)$, $p(\Gamma)$ are characterized. Furthermore, we provide all extremal signed graphs for the lower bound. At first, we consider the connected signed graphs and the following inequalities hold.
\begin{thm}\label{thm-3-1}
Let $\Gamma$ be a connected signed graph of order $n\geq 2$. Then $i_+(\Gamma)\geq \frac{n-p(\Gamma)}{2}-\theta(\Gamma)$. Moreover, $i_+(\Gamma)\geq \frac{n-p(\Gamma)+1}{2}-\theta(\Gamma)$ if $p(\Gamma)\geq 1$, or $p(\Gamma)=0$ and two distinct cycles of $\Gamma$ have common vertices.
\end{thm}


\begin{proof} We will prove this theorem by induction on the order $n$ of $\Gamma$. Firstly, the cases $n=2$ and $n=3$ are verified. If $n=2$, then $\Gamma\cong P_2^\sigma$. Note that $i_+(P_2^\sigma)=1$, $\theta(P_2^\sigma)=0$, and $p(P_2^\sigma)=2$. We have $1=i_+(P_2^\sigma)\geq \frac{2-p(P_2^\sigma)+1}{2}-\theta(P_2^\sigma)=\frac{1}{2}$.
If $n=3$, then $\Gamma\cong P_3^\sigma$ or $C_3^\sigma$ since $\Gamma$ is connected. Suppose that $\Gamma\cong P_3^\sigma$. Then $1=i_+(P_3^\sigma)\geq \frac{3-p(P_3^\sigma)+1}{2}-\theta(P_3^\sigma)=1$ by Lemma \ref{lem-2-5} (iii). Suppose that $\Gamma\cong C_3^\sigma$. Then $1=i_+(C_3^\sigma)\geq \frac{3-p(C_3^\sigma)}{2}-\theta(C_3^\sigma)=\frac{1}{2}$ if $C_3^\sigma$ is balanced, and $2=i_+(C_3^\sigma)\geq \frac{3-p(C_3^\sigma)}{2}-\theta(C_3^\sigma)=\frac{1}{2}$ if $C_3^\sigma$ is unbalanced. Hence it suffice to consider the case $n\geq 4$ as follows and we may assume that the inequalities hold for any connected graph $H^\sigma$ with $2\leq |V(H^\sigma)|\leq n-1$. The following three cases can be distinguished.

{\flushleft\bf Case 1.} $\Gamma$ has no pendant vertices.

Clearly, $\Gamma$ contains at least one cycle, say $C^\sigma$, in this case. Let $x\in V(C^\sigma)$ and $\Gamma-x=H_1^\sigma\cup H_2^\sigma\cup\cdots\cup H_s^\sigma$, where $H_k^\sigma$ $(k=1,\ldots,s)$ are all connected components of $\Gamma-x$. Since $\Gamma$ has no pendant vertices, $2\leq |V(H_k^\sigma)|\leq n-1$ for any $k$. We divide these components into two categories such that $H_i^\sigma$ $(i=1, \ldots, r)$ has pendant vertices and $H_j^\sigma$ $(j=r+1,\ldots,s)$ has no pendant vertices. By applying induction hypothesis to each component $H_i^\sigma$ $(i=1, \ldots, r)$ and $H_j^\sigma$ $(j=r+1,\ldots,s)$, we have
$$ i_+(H_i^\sigma)\geq\frac{n(H_i^\sigma)-p(H_i^\sigma)+1}{2}-\theta(H_i^\sigma),  \ i=1,2,\ldots,r$$
and
$$ i_+(H_j^\sigma)\geq\frac{n(H_j^\sigma)}{2}-\theta(H_j^\sigma),\ j=r+1,\ldots,s.$$
Hence
\begin{align}
i_+(\Gamma)\geq i_+(\Gamma-x)&=\sum_{i=1}^{r}i_+(H_i^\sigma)+\sum_{j=r+1}^{s}i_+(H_j^\sigma)\nonumber\\
&\geq\sum_{i=1}^{r}(\frac{n(H_i^\sigma)-p(H_i^\sigma)+1}{2}-\theta(H_i^\sigma))+\sum_{j=r+1}^{s}(\frac{n(H_j^\sigma)}{2}-\theta(H_j^\sigma))\nonumber\\
&=\sum_{k=1}^{s}\frac{n(H_k^\sigma)}{2}-\sum_{i=1}^{r}\frac{p(H_i^\sigma)}{2}-\sum_{k=1}^{s}\theta(H_k^\sigma)+\frac{r}{2}\nonumber\\
&=\frac{n(\Gamma-x)-p(\Gamma-x)}{2}-\theta(\Gamma-x)+\frac{r}{2}.\nonumber
\end{align}
by Lemma \ref{lem-2-1} and Lemma \ref{lem-2-3}.

Let $m$ be the number of 2-degree neighbors of $x$. Since $\Gamma$ has no pendant vertices, the number of pendant vertices of $\Gamma-x$ equals $m$, that is, $p(\Gamma-x)=\sum_{i=1}^{r}p(H_i^\sigma)=m$. On the other hand, in light of Lemma \ref{lem-2-4} (ii), we have $\theta(\Gamma-x) =\theta(\Gamma)-d(x)+s$. Hence
$$i_+(\Gamma)\geq\frac{n-1-m}{2}-(\theta(\Gamma)-d(x)+s)+\frac{r}{2}=\frac{n}{2}-\theta(\Gamma)+(d(x)-s-\frac{m}{2}+\frac{r}{2}-\frac{1}{2}).$$
Note that $d(x)\geq s+1$ since $x\in V(C^\sigma)$. The inequality $2d(x)\geq 2s+m-r+1$ holds by Lemma \ref{lem-2-4} (i). Hence $d(x)-s-\frac{m}{2}+\frac{r}{2}-\frac{1}{2}\geq 0$, and so
$$i_+(\Gamma)\geq\frac{n}{2}-\theta(\Gamma)=\frac{n-p(\Gamma)}{2}-\theta(\Gamma)$$ because $p(\Gamma)=0$, as desired.

Furthermore, suppose that two distinct cycles $C_1^\sigma$ and $C_2^\sigma$ have common vertices in $\Gamma$. Choose $x$ in such a way that $x\in V(C_1^\sigma)\cap V(C_2^\sigma)$ and $d_{C_1^\sigma\cup C_2^\sigma}(x)\geq 3$, which implies $d(x)\geq s+2$. Hence $2d(x)\geq 2s+m-r+2$ by Lemma \ref{lem-2-4} (i), and so $d(x)-s-\frac{m}{2}+\frac{r}{2}-\frac{1}{2}\geq \frac{1}{2}$. Recall that $p(\Gamma)=0$, the inequality
$$i_+(\Gamma)\geq(\frac{n}{2}-\theta(\Gamma))+\frac{1}{2}=\frac{n-p(\Gamma)+1}{2}-\theta(\Gamma)$$ holds.

{\flushleft\bf Case 2.} $\Gamma$ has at least two pendant vertices.

Suppose that $x$ is one of pendant vertices of $\Gamma$ and let $y$ be the neighbour of $x$. If $d(y)>2$, then $\theta(\Gamma-x)=\theta(\Gamma)$ and $p(\Gamma-x)=p(\Gamma)-1\geq 1$. Using the fact that $i_+(\Gamma)\geq i_+(\Gamma-x)$ by Lemma \ref{lem-2-1}. By applying induction hypothesis to $\Gamma-x$, the inequaliy
\begin{align}
i_+(\Gamma)&\geq i_+(\Gamma-x)\nonumber\\
&\geq \frac{n(\Gamma-x)-p(\Gamma-x)+1}{2}-\theta(\Gamma-x)\nonumber\\
&=\frac{n-p(\Gamma)+1}{2}-\theta(\Gamma)\nonumber
\end{align}
holds.

If $d(y)=2$, then $1\leq p(\Gamma)-1\leq p(\Gamma-x-y)\leq p(\Gamma)$ since $n\geq 4$. Note that $\theta(\Gamma-x-y)=\theta(\Gamma)$. Again by applying induction hypothesis to $\Gamma-x-y$, we have
\begin{align}
i_+(\Gamma)&=i_+(\Gamma-x-y)+1\nonumber\\
&\geq\frac{n(\Gamma-x-y)-p(\Gamma-x-y)+1}{2}-\theta(\Gamma-x-y)+1\nonumber\\
&\geq\frac{n-2-p(\Gamma)+1}{2}-\theta(\Gamma)+1\nonumber\\
&=\frac{n-p(\Gamma)+1}{2}-\theta(\Gamma)\nonumber
\end{align}
by Lemma \ref{lem-2-2}.

{\flushleft\bf Case 3.} $\Gamma$ has one unique pendant vertex.

Suppose that $x$ is the unique pendant vertex of $\Gamma$ and let $y$ be the neighbour of $x$. By virtue of Lemma \ref{lem-2-2}, we have
$$i_+(\Gamma)=i_+(\Gamma-x-y)+1.$$

At first, let us consider the subcase $d(y)=2$ and set $x\neq z\in N(y)$. If $d(z)=2$ in this case, then $p(\Gamma-x-y)=p(\Gamma)=1$, $\theta(\Gamma-x-y)=\theta(\Gamma)$. By applying induction hypothesis to $\Gamma-x-y$, we obtain
\begin{align}
i_+(\Gamma)&=i_+(\Gamma-x-y)+1\nonumber\\
&\geq\frac{n(\Gamma-x-y)-p(\Gamma-x-y)+1}{2}-\theta(\Gamma-x-y)+1\nonumber\\
&=\frac{n-2-p(\Gamma)+1}{2}-\theta(\Gamma)+1=\frac{n-p(\Gamma)+1}{2}-\theta(\Gamma).\nonumber
\end{align}
If $d(z)\geq3$, then $p(\Gamma-x-y)=p(\Gamma)-1=0, \theta(\Gamma-x-y)=\theta(\Gamma)$. In this case, we obtain
\begin{align}
i_+(\Gamma)&=i_+(\Gamma-x-y)+1\nonumber\\
&\geq\frac{n(\Gamma-x-y)-p(\Gamma-x-y)}{2}-\theta(\Gamma-x-y)+1\nonumber\\
&=\frac{n-2-p(\Gamma)+1}{2}-\theta(\Gamma)+1=\frac{n-p(\Gamma)+1}{2}-\theta(\Gamma).\nonumber
\end{align}

Now, it suffice to prove the inequality when $d(y)\geq 3$. Let $(\Gamma-y)-x=H_1^\sigma\cup H_2^\sigma\cup\cdots\cup H_s^\sigma$ be disjoint union of connected signed components of $(\Gamma-y)-x$. By the similar way as Case 1, set $H_i^\sigma$ $(i=1,\cdots,r)$ contains 2-degree neighbors of $y$, and $H_j^\sigma$ $(j=r+1,\cdots,s)$ contains no 2-degree neighbors of $y$. Since $x$ is the unique pendant vertex of $\Gamma$, we have $2 \leq |V(H_k^\sigma)|\leq n-1$ for each $k\in \{1,\ldots, s\}$. By applying induction hypothesis to each component $H_k$, the following inequalities hold:
$$i_+(H_i^\sigma)\geq\frac{n(H_i^\sigma)-p(H_i^\sigma)+1}{2}-\theta(H_i^\sigma),\ i=1,2,\ldots,r,$$
and
$$i_+(H_j^\sigma)\geq\frac{n(H_j^\sigma)}{2}-\theta(H_j^\sigma),\ j=r+1,\ldots,s.$$

It is clear that $p((\Gamma-y)-x)=\sum_{i=1}^{r}p(H_i^\sigma)=m$ holds, where $m$ denotes the number of 2-degree neighbors of $y$ in $\Gamma$. Hence
\begin{align}
i_+(\Gamma)&=i_+((\Gamma-y)-x)+1=\sum_{i=1}^{r}i_+{(H_i^\sigma)}+\sum_{j=r+1}^{s}i_+{(H_j^\sigma)}+1\nonumber\\
&\geq\sum_{i=1}^{r}(\frac{n(H_i^\sigma)-p(H_i^\sigma)+1}{2}-\theta(H_i^\sigma))+\sum_{j=r+1}^{s}(\frac{n(H_j^\sigma)}{2}-\theta(H_j^\sigma))+1\nonumber\\
&=\sum_{k=1}^{s}\frac{n(H_k^\sigma)}{2}-\sum_{i=1}^{r}\frac{p(H_i^\sigma)}{2}+\frac{r}{2}-\sum_{k=1}^{s}\theta(H_k^\sigma)+1\nonumber\\
&=\frac{n-2-m}{2}-\theta((\Gamma-y)-x)+\frac{r}{2}+1\nonumber.
\end{align}
by Lemma \ref{lem-2-3}.

Note that $y$ must belong to some cycle of $\Gamma$ because $d(y)\geq 3$ and $x$ is unique pendant vertex of $G$. We have $d(y)\geq s+1$. Hence $2d(y)\geq 2s+m-r+1$ by Lemma 2.4 (i). On the other hand, from Lemma \ref{lem-2-4} (ii), one may deduce the equality $\theta((\Gamma-y)-x)=\theta(\Gamma-y)=\theta(\Gamma)-d(y)+s$. Consequently,
\begin{align}
i_+(\Gamma)&\geq(\frac{n+1}{2}-\theta(\Gamma))+(d(y)-s-\frac{m}{2}+\frac{r}{2}-\frac{1}{2})\nonumber\\
&\geq\frac{n-p(\Gamma)+1}{2}-\theta(\Gamma)\nonumber
\end{align}
since $p(\Gamma)=0$. The theorem follows by induction.
\end{proof}

\begin{remark}\label{re-1}
\emph{For a disconnected signed graph $\Gamma$, in light of Lemma \ref{lem-2-3} and the additive property of $p(\Gamma)$ and $\theta(\Gamma)$, the inequalities of Theorem \ref{thm-3-1} can therefore be restricted or reduced to each connected signed component of $\Gamma$. Hence these inequalities also hold for disconnected signed graphs.}
\end{remark}

Let $\Gamma$ be a cycle-disjoint signed graph, and $C_1^\sigma, \ldots, C_k^\sigma$ be all induced signed cycles of $\Gamma$. In general, the signed cycles of $\Gamma$ fit together in a treelike structure. By $\Gamma\setminus \{C_1^\sigma,\ldots, C_k^\sigma\}$ we denote the graph obtained from $\Gamma$ by \emph{contracting} each $C_i^\sigma$ $(i=1,\ldots,k)$ into a new vertex $v_{C_i^\sigma}$. Moreover, $v_{C_i^\sigma}$ is adjacent to, except for the former neighbours in $\Gamma$, the vertex $v_{C_j^\sigma}$ for which some vertices of $C_i^\sigma$ are adjacent the vertex of $C_j^\sigma$ in $\Gamma$. Clearly, $\Gamma\setminus \{C_1^\sigma,\ldots, C_k^\sigma\}$ must be a tree since $\Gamma$ is cycle-disjoint. For example, as illustrated in Fig. \ref{fig-1}, for a cycle-disjoint signed graph $\Gamma$, we provide the tree $\Gamma\setminus \{C_1^\sigma,\ldots,C_6^\sigma\}$, which can be denoted by $T_{\Gamma}$. At the end of this section, by using the tree structure introduced above, we will give the characterization of signed graph $\Gamma$ with $i_+(\Gamma)=\frac{n-p(\Gamma)}{2}-\theta(\Gamma)$.

\begin{figure}[t]
    \centering
    \includegraphics[width=0.75\linewidth]{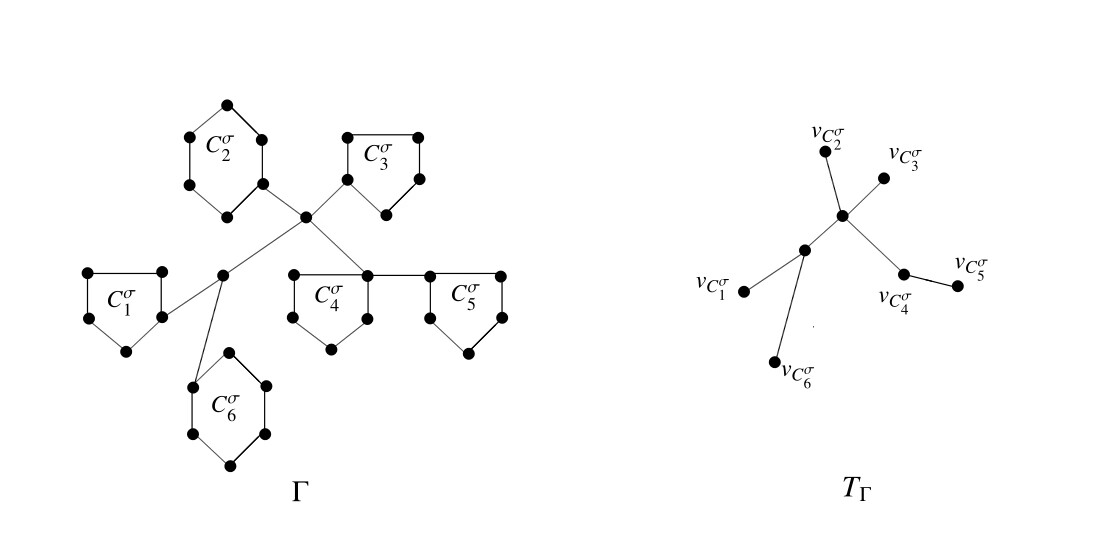}
    \caption{The tree $T_{\Gamma}$ corresponding to cycle-disjoint signed graph $\Gamma$.}
    \label{fig-1}
    \end{figure}

\begin{thm}\label{thm-3-2}
Let $\Gamma$ be a signed graph of order $n$ and each connected component of $\Gamma$ has at least two vertices. Then $i_+(\Gamma)=\frac{n-p(\Gamma)}{2}-\theta(\Gamma)$ if and only if $\Gamma\cong C_{n_1}^\sigma\cup C_{n_2}^\sigma\cup\cdots \cup C_{n_t}^\sigma$, where $n_1+n_2+\cdots+n_t=n$, and $n_i\equiv 0(\bmod\ 4)$ if $C_{n_i}^\sigma$ is balanced or $n_i\equiv 2(\bmod\ 4)$ if $C_{n_i}^\sigma$ is unbalanced for $i\in \{1,2,\ldots,t\}$.
\end{thm}
\begin{proof}
Suppose that $\Gamma\cong C_{n_1}^\sigma\cup C_{n_2}^\sigma\cup \cdots \cup C_{n_t}^\sigma$, where $n_1+n_2+\cdots+n_t=n$ and $n_i\equiv 0(\bmod\ 4)$ if $C_{n_i}^\sigma$ is balanced or $n_i\equiv 2(\bmod\ 4)$ if $C_{n_i}^\sigma$ is unbalanced. Then $i_+(\Gamma)=(\frac{n_1}{2}-1)+(\frac{n_2}{2}-1)+\cdots+ (\frac{n_t}{2}-1)=\frac{n}{2}-t$ by Lemma \ref{lem-2-5}. Since $p(\Gamma)=0$ and $\theta(\Gamma)=t$, the equality $i_+(\Gamma)=\frac{n-p(\Gamma)}{2}-\theta(\Gamma)$ holds.

Conversely, let $\Gamma$ be a signed graph satisfying $i_+(\Gamma)=\frac{n-p(\Gamma)}{2}-\theta(\Gamma)$. Then $i_+(H_i^\sigma)=\frac{n_i-p(H_i^\sigma)}{2}-\theta(H_i^\sigma)$ for each connected signed component $H_i^\sigma$ of $\Gamma$, where $n_i=n(H_i^\sigma)$. Clearly, by virtue of Remark \ref{re-1}, we assert that $p(H_i^\sigma)=0$ and any distinct cycles of $H_i^\sigma$ share no vertices in common. Our first task is to demonstrate $H_i^\sigma\cong C_{n_i}^\sigma$. On the contrary, assume that $H_i^\sigma$ contains at least two signed cycles. Contracting each induced signed cycle of $H_i^\sigma$ into a new vertex, we obtain a tree $T_{H_i^\sigma}$ by the above assertion. Hence there exists a signed cycle $C^\sigma$ of $H_i^\sigma$, which corresponding to a pendant vertex of $T_{H_i^\sigma}$, such that exactly one vertex $x$ in $V(C^\sigma)$ satisfies $d(x)=3$.

Let $y$ be a neighbor of $x$ in $C^\sigma$. Then $H_i^\sigma-y$ has exactly one pendant vertex and $\theta(H_i^\sigma-y)=\theta(H_i^\sigma)-1$. Since $p(H_i^\sigma)=0$, we have
\begin{align}
i_+(H_i^\sigma)\geq i_+(H_i^\sigma-y)&\geq\frac{n(H_i^\sigma-y)-p(H_i^\sigma-y)+1}{2}-\theta(H_i^\sigma-y)\nonumber\\
&=\frac{n_i-1-1+1}{2}-\theta(H_i^\sigma)+1\nonumber\\
&=(\frac{n_i-p(H_i^\sigma)}{2}-\theta(H_i^\sigma))+\frac{1}{2}\nonumber
\end{align}
by Lemma \ref{lem-2-1} and Theorem \ref{thm-3-1}. This contradicts the assumption $i_+(H_i^\sigma)=\frac{n_i-p(H_i^\sigma)}{2}-\theta(H_i^\sigma)$. Hence $H_i^\sigma$ has exactly one induced signed cycle, and so $H_i^\sigma\cong C_{n_i}^\sigma$ because of $p(H_i^\sigma)=0$. Recall that $i_+(H_i^\sigma)=\frac{n_i-p(H_i^\sigma)}{2}-\theta(H_i^\sigma)=\frac{n_i}{2}-1$. We obtain that $n_i\equiv 0(\bmod\ 4)$ if $H_i^\sigma$ is balanced or $n_i\equiv 2(\bmod\ 4)$ if $H_i^\sigma$ is unbalanced by Lemma \ref{lem-2-5}.

\end{proof}

\section{Negative inertia index and nullity of signed graphs with given cyclomatic number and given number of pendant vertices}
By replacing the signature $\sigma$ of $\Gamma=(G, \sigma)$ with $-\sigma$, one can obtain a signed graph $-\Gamma=(G, -\sigma)$ known as the negation of $\Gamma$. By elementary linear algebra, the negative inertia index of $\Gamma$ is equal to the positive inertia index of $-\Gamma$. Hence the following Theorems \ref{thm-4-1} and \ref{thm-4-2} in this draft is an immediate consequence of Theorems \ref{thm-3-1} and \ref{thm-3-2}. To better understand these two theorems, we also provide a brief proof process. But the proofs of results about $i_-(\Gamma)$ are similar to that of $i_+(\Gamma)$, we will omit some similar steps.

\begin{thm}\label{thm-4-1}
Let $\Gamma$ be a connected signed graph of order $n\geq 2$. Then $i_-(\Gamma)\geq \frac{n-p(\Gamma)}{2}-\theta(\Gamma)$. Moreover, $i_-(\Gamma)\geq \frac{n-p(\Gamma)+1}{2}-\theta(\Gamma)$ if $p(\Gamma)\geq 1$, or $p(\Gamma)=0$ and two distinct cycles of $\Gamma$ have common vertices.
\end{thm}
\begin{proof}
We also prove this theorem by induction on $n$. Firstly, the inequalities hold for $n=2$ or $n=3$ obviously. Hence let $n\geq 4$ and we discuss the following three cases.

{\flushleft\bf Case 1.} $\Gamma$ has no pendant vertices.

$\Gamma$ must contain an induced signed cycle $C^\sigma$ in this case. Let $x\in V(C^\sigma)$ and $\Gamma-x=H_1^\sigma\cup H_2^\sigma\cup\cdots\cup H_s^\sigma$, where $H_k^\sigma$ $(k=1,\ldots,s)$ are all connected components of $\Gamma-x$. Then $|V(H_k^\sigma)|\geq 2$ since $\Gamma$ has no pendant vertices. Set $H_i^\sigma$ $(i=1, \ldots, r)$ has pendant vertex and $H_j^\sigma$ $(j=r+1,\ldots,s)$ has no pendant vertices. By applying induction hypothesis to $H_i^\sigma$ and $H_j^\sigma$, we have
$i_-(H_i^\sigma)\geq\frac{n(H_i^\sigma)-p(H_i^\sigma)+1}{2}-\theta(H_i^\sigma)$ and $i_-(H_j^\sigma)\geq\frac{n(H_j^\sigma)}{2}-\theta(H_j^\sigma)$. Hence
$$i_-(\Gamma)\geq i_-(\Gamma-x)=\sum_{k=1}^{s}i_-(H_k^\sigma)\geq\frac{n(\Gamma-x)-p(\Gamma-x)}{2}-\theta(\Gamma-x)+\frac{r}{2}$$
by Lemmas \ref{lem-2-1} and \ref{lem-2-3}. Note that $d(x)\geq s+1$ since $x\in V(C^\sigma)$. We have $\theta(\Gamma-x)=\theta(\Gamma)-d(x)+s$ and $2d(x)\geq 2s+m-r+1$ by Lemma \ref{lem-2-4}. Hence $i_-(\Gamma)\geq\frac{n-p(\Gamma)}{2}-\theta(G)$ since $p(\Gamma)=0$.

Furthermore, if two distinct cycles $C_1^\sigma$ and $C_2^\sigma$ of $\Gamma$ have common vertices, then there exists $x\in V(C_1^\sigma)\cap V(C_2^\sigma)$ such that $d(x)\geq s+2$. Hence $2d(x)\geq 2s+m-r+2$ by Lemma \ref{lem-2-4} (i). Similar to positive inertia index, we have $i_-(\Gamma)\geq\frac{n-p(\Gamma)+1}{2}-\theta(\Gamma)$.

{\flushleft\bf Case 2.} $\Gamma$ has at least two pendant vertices.

Let $x\in V(\Gamma)$ be a pendant vertex and $N(x)=\{y\}$. Suppose that $d(y)>2$. By applying induction hypothesis to $\Gamma-x$, we have
$$i_-(\Gamma)\geq i_-(\Gamma-x)\geq\frac{n(\Gamma-x)-p(\Gamma-x)+1}{2}-\theta(\Gamma-x)=\frac{n-p(\Gamma)+1}{2}-\theta(\Gamma)$$ by Lemma \ref{lem-2-1}. Suppose that $d(y)=2$. By applying induction hypothesis to $\Gamma-x-y$, we have
$$i_-(\Gamma)=i_-(\Gamma-x-y)+1\geq\frac{n(\Gamma-x-y)-p(\Gamma-x-y)+1}{2}-\theta(\Gamma-x-y)+1\geq\frac{n-p(\Gamma)+1}{2}-\theta(\Gamma).$$
by Lemma \ref{lem-2-2}.

{\flushleft\bf Case 3.} $\Gamma$ has one unique pendant vertex.

Let $x$ be the unique pendant vertex of $\Gamma$ and $N(x)=\{y\}$. Suppose that $d(y)=2$ and set $z\in N(y)$. Then $\Gamma-x-y$ has a pendant vertex if $d(z)=2$, and $\Gamma-x-y$ has no pendant vertices if $d(z)\geq 3$. By applying induction hypothesis to $\Gamma-x-y$ in both cases, we obtain
$i_-(\Gamma)=i_-(\Gamma-x-y)+1\geq\frac{n-p(\Gamma)+1}{2}-\theta(\Gamma)$ by Lemma \ref{lem-2-2}.

Suppose that $d(y)\geq 3$ and let $\Gamma-y-x=H_1^\sigma\cup H_2^\sigma\cup\cdots\cup H_s^\sigma$. Set $H_i^\sigma$ $(i=1,\ldots,r)$ contains 2-degree neighbors of $y$, and $H_j^\sigma$ $(j=r+1,\ldots,s)$ contains no 2-degree neighbors of $y$. By applying induction hypothesis to $H_i^\sigma$ and $H_j^\sigma$, we have
$i_-(H_i^\sigma)\geq\frac{n(H_i^\sigma)-p(H_i^\sigma)+1}{2}-\theta(H_i^\sigma)$ and $i_-(H_j^\sigma)\geq\frac{n(H_j^\sigma)}{2}-\theta(H_j^\sigma)$.
Since $p(\Gamma-y-x)=\sum_{i=1}^{r}p(H_i^\sigma)=m$, where $m$ denotes the number of 2-degree neighbors of $y$ in $G$, we obtain
$$i_-(\Gamma)=i_-(\Gamma-y-x)+1=\frac{n-2-m}{2}-\theta(\Gamma-y-x)+\frac{r}{2}+1$$
by Lemmas \ref{lem-2-2} and \ref{lem-2-3}. According to Lemma \ref{lem-2-4}, the inequalities $2d(y)\geq 2s+m-r+1$ and $\theta(\Gamma-y-x)=\theta(\Gamma)-d(y)+s$ holds since $d(y)\geq s+1$. Therefore,
$i_-(\Gamma)\geq(\frac{n+1}{2}-\theta(\Gamma))+(d(y)-s-\frac{m}{2}+\frac{r}{2}-\frac{1}{2})\geq\frac{n-p(\Gamma)+1}{2}-\theta(\Gamma)$.
\end{proof}

\begin{remark}\label{re-2}
Similar to the positive inertia index, Theorem \ref{thm-4-1} also holds for disconnected signed graphs.
\end{remark}

\begin{thm}\label{thm-4-2}
Let $\Gamma$ be a signed graph of order $n$ and each connected component of $\Gamma$ has at least two vertices. Then $i_-(\Gamma)=\frac{n-p(\Gamma)}{2}-\theta(\Gamma)$ if and only if $\Gamma\cong C_{n_1}^\sigma\cup C_{n_2}^\sigma\cup\cdots \cup C_{n_t}^\sigma$, where $n_1+n_2+\cdots+n_t=n$, and $n_i\equiv 0(\bmod\ 4)$ if $C_{n_i}^\sigma$ is balanced or $n_i\equiv 2(\bmod\ 4)$ if $C_{n_i}^\sigma$ is unbalanced for $i\in \{1,2,\ldots,t\}$.
\end{thm}
\begin{proof}
The sufficiency holds obviously by Lemma \ref{lem-2-5}. We prove the necessity as follows. It only need to consider the connected component $H_i^\sigma$ of $\Gamma$ which satisfies $i_-(H_i^\sigma)=\frac{n_i-p(H_i^\sigma)}{2}-\theta(H_i^\sigma)$, where $n_i=n(H_i^\sigma)$. According to Theorem \ref{thm-4-1}, we assert that $p(H_i^\sigma)=0$ and any distinct cycles of $H_i^\sigma$ share no vertices in common. Firstly, we claim $H_i^\sigma\cong C_{n_i}^\sigma$ as follows. On the contrary, assume that $H_i^\sigma$ contains at least two cycles. Then there exists a signed cycle $C^\sigma$ of $H_i^\sigma$ such that exactly one vertex, say $x\in V(C^\sigma)$, is the only vertex with $d(x)=3$. Let $y$ be a neighbor of $x$ in $C^\sigma$. Then $H_i^\sigma-y$ has exactly one pendant vertex and $\theta(H_i^\sigma-y)=\theta(H_i^\sigma)-1$. Note that $p(H_i^\sigma)=0$, we have
$$i_-(H_i^\sigma)\geq i_-(H_i^\sigma-y)=(\frac{n-p(H_i^\sigma)}{2}-\theta(H_i^\sigma))+\frac{1}{2}$$
by Lemma \ref{lem-2-1} and Theorem \ref{thm-4-1}, a contradiction. Hence $H_i^\sigma\cong C_{n_i}^\sigma$, and so $n_i\equiv 0\ (\bmod\ 4)$ if $H_i^\sigma$ is balanced or $n_i\equiv 2\ (\bmod\ 4)$ if $H_i^\sigma$ is unbalanced by Lemma \ref{lem-2-5}.
\end{proof}

By using the inequalities of positive and negative inertia indices of a signed $\Gamma$ shown in Theorems \ref{thm-3-1} and \ref{thm-4-1}, we can obtain following corollary about the nullity of $\Gamma$, which contains main conclusions in \cite{Xiaobin.Ma}.

\begin{cor}\label{cor-4-1}
Let $\Gamma$ be a signed graph of order $n\geq 2$. Then $\eta(\Gamma)\leq p(\Gamma)+2\theta(\Gamma)$. Moreover, $\eta(\Gamma)\leq p(\Gamma)+2\theta(\Gamma)-1$ if $p(\Gamma)\geq 1$, or $p(\Gamma)=0$ and two distinct cycles of $G$ have common vertices.
\end{cor}
\begin{proof}
Let $\Gamma$ be a signed graph of order $n\geq 2$. In light of Remark \ref{re-1} and Remark \ref{re-2}, we have
$$r(\Gamma)=i_+(\Gamma)+i_-(\Gamma)\geq(\frac{n-p(\Gamma)}{2}-\theta(\Gamma))+(\frac{n-p(\Gamma)}{2}-\theta(\Gamma))=n-p(\Gamma)-2\theta(\Gamma).$$
Therefore,
$$\eta(\Gamma)=n-r(\Gamma)\leq n-(n-p(\Gamma)-2\theta(\Gamma))=p(\Gamma)+2\theta(\Gamma).$$

Furthermore, if $p(\Gamma)\geq 1$, or $p(\Gamma)=0$ and two distinct cycles of $\Gamma$ have common vertices, then

$$r(\Gamma)=i_+(\Gamma)+i_-(\Gamma)\geq(\frac{n-p(\Gamma)+1}{2}-\theta(\Gamma))+(\frac{n-p(\Gamma)+1}{2}-\theta(\Gamma))=n-p(\Gamma)-2\theta(\Gamma)+1.$$
Hence
$$\eta(\Gamma)=n-r(\Gamma)\leq n-(n-p(\Gamma)-2\theta(\Gamma)+1)=p(\Gamma)+2\theta(\Gamma)-1.$$
This complete the proof.
\end{proof}

\begin{cor}\label{cor-4-2}
Let $\Gamma$ be a signed graph of order $n$ and each connected component of $\Gamma$ has at least two vertices. Then $\eta(\Gamma)=p(\Gamma)+2\theta(\Gamma)$ if and only if  $\Gamma\cong C_{n_1}^\sigma\cup C_{n_2}^\sigma\cup\cdots \cup C_{n_t}^\sigma$, where $n_1+n_2+\cdots+n_t=n$, and $n_i\equiv 0(\bmod\ 4)$ if $C_{n_i}^\sigma$ is balanced or $n_i\equiv 2(\bmod\ 4)$ if $C_{n_i}^\sigma$ is unbalanced for $i\in \{1,2,\ldots,t\}$.
\end{cor}
\begin{proof}
According to Theorems \ref{thm-3-2} and \ref{thm-4-2}, the sufficiency of this corollary holds by the equality $\eta(\Gamma)=n-(i_+(\Gamma)+i_-(\Gamma))$.

Conversely, let $\Gamma$ be a signed graph satisfying $\eta(\Gamma)=p(\Gamma)+2\theta(\Gamma)$. Then $r(\Gamma)=n-p(\Gamma)-2\theta(\Gamma)$. We claim that $i_+(\Gamma)=\frac{n-p(\Gamma)}{2}-\theta(\Gamma)$ must be correct in this context. In fact, on the contrary, assume that $i_+(\Gamma)\not=\frac{n-p(\Gamma)}{2}-\theta(\Gamma)$. Then $i_+(\Gamma)>\frac{n-p(\Gamma)}{2}-\theta(\Gamma)$ by Theorem \ref{thm-3-1}. Hence
$i_-(\Gamma)=r(\Gamma)-i_+(\Gamma)<\frac{n-p(\Gamma)}{2}-\theta(\Gamma)$, which contradicts the conclusion of Theorem \ref{thm-4-1}. Therefore, $i_+(\Gamma)=\frac{n-p(\Gamma)}{2}-\theta(\Gamma)$, and so the assertion holds by Theorem \ref{thm-3-2}.
\end{proof}

\noindent \textbf{Declaration of competing interest}\vspace{3mm}

The authors declare that there are no conflict of interests.\vspace{3mm}

\noindent \textbf{Data availability}\vspace{3mm}

No data was used for the research described in the article.

\end{document}